\documentclass[fullpage]{article}

\usepackage{amsmath,amsthm,amsfonts,amssymb,amscd,epsf,epsfig,psfrag,enumerate, tikz,bbold}
\usepackage[linesnumbered]{algorithm2e}
\usepackage{authblk}

\newcommand {\Z}	  {\mathbb{Z}}
\newcommand {\Q}	  {\mathbb{Q}}
\newcommand {\R}	  {\mathbb{R}}

\newcommand {\st}	  {\text{s.t.}}
\newcommand{\deltabound}{\delta}
\newcommand{\entrybound}{C}
\newcommand{\rank}{\mathrm{rank}}
\newcommand{\maxsubdet}{\deltabound_{\max}}

\renewcommand{\epsilon}{\varepsilon}

\renewcommand{\leq}{\leqslant}
\renewcommand{\geq}{\geqslant}

\theoremstyle{plain}
\newtheorem{theorem}{Theorem}
\newtheorem{lemma}[theorem]{Lemma}
\newtheorem{remark}[theorem]{Remark}

\title{\bf A Note on Non-Degenerate Integer Programs with Small Sub-Determinants}
\author[1]{S.~Artmann}
\author[2]{F.~Eisenbrand}
\author[1]{C.~Glanzer}
\author[3]{T.~Oertel}
\author[4]{S.~Vempala}
\author[1]{R.~Weismantel}
\affil[1]{Swiss Federal Institute of Technology, Z\"urich (ETH Z\"urich)}
\affil[2]{\'Ecole Polytechnique F\'ed\'erale de Lausanne (EPFL)}
\affil[3]{Cardiff University}
\affil[4]{Georgia Institute of Technology}

\begin{document}


\maketitle
%
%

\begin{abstract}
\noindent The intention of this note is two-fold. First, we study integer optimization problems in standard form defined by $A \in\Z^{m\times{}n}$ and present an algorithm to solve such problems in polynomial-time provided that both the largest absolute value of an entry in $A$ and $m$ are constant. Then, this is applied to solve integer programs in inequality form in polynomial-time, where the absolute values of all maximal sub-determinants of $A$ lie between $1$ and a constant.
\end{abstract}

%

%
	
\section{Introduction}

Let $A \in \Z^{m \times n}$ be a matrix such that all of its entries are bounded in absolute value by an integer $\Delta$. Assume that for each row index $i$, $\gcd(A_{i,\cdot})=1$. We call the determinant of an $(n\times n)$-submatrix of $A$ an $(n\times n)$-sub-determinant of $A$. Let
\begin{align*}
\maxsubdet(A) :=& \max\{|d| \colon d \text{ is an $(n\times n)$-sub-determinant of }A \}.
\end{align*}  We study the complexity of an integer programming problem in terms of the parameter  $\Delta$ when presented in standard form \eqref{eq:1}. Moreover,   we study integer programming problems  in inequality form (\ref{equ:opt_problem}) that are associated with the matrix $A$ whose `complexity' is measured by the parameter $\maxsubdet(A)$.
\begin{align}
\max & \left\{ c^Tx \colon Ax = b, \, x\geq0, \,x \in \Z^n\right\}, \label{eq:1}\\
\max & \left\{c^Tx \colon Ax\le b, x \in \Z^n \right\}.\label{equ:opt_problem}
\end{align} 
It is known that when the absolute value of all sub-determinants of $A$ is bounded by one, $A$ is totally unimodular and the integer programs \eqref{eq:1} and \eqref{equ:opt_problem} are polynomially solvable. This concept of total unimodularity was pioneered by the works of Hoffman, Kruskal, Veinott, Dantzig and many other researchers. It has led to a beautiful and fundamental theory  that is so important that it is  covered by all standard textbooks in combinatorial optimization nowadays. For instance, see \cite{Schrijver86} for a thorough treatment of the subject.

When $\maxsubdet(A) > 1$, then surprisingly little is known.

Bonifas et al.~showed in \cite{Bonifasetal} that for a bounded polyhedron $P=\{x \in \R^n \colon Ax \leq b\}$ its (combinatorial) diameter is bounded in the order of $O(\maxsubdet(A)^2 \cdot n^{3.5} \cdot \log( n \cdot \maxsubdet(A)))$. This improves an important result of Dyer and Frieze \cite{Dyeretal} that applies to TU-matrices.

Veselov and Chirkov (2009) showed in \cite{veselov2009Intprobimmat}, how \eqref{equ:opt_problem} can be solved polynomially in $m$ and $n$  and the encoding size of the data when $\maxsubdet(A)\le 2$ and no $(n \times n)$-submatrices are singular.

There exists  
a dynamic programming approach to solve~\eqref{eq:1} by 
Papadimitriou~\cite{MR677087}, see also~\cite{Schrijver86}, Part IV, Section 18.6: Let $\Delta(A,b)$
be an upper bound on the absolute values of $A$ and
$b$. Then, if~\eqref{eq:1} is feasible and bounded, it has
an optimal solution with  components bounded by $U := (n+1) (m \cdot
\Delta(A,b))^m$.   

The dynamic program is a maximum weight path problem on a properly defined (acyclic)
graph. The optimum
solution can be found in time 
\begin{equation*}
\mathcal{O} ( | V| ) = \mathcal{O}(2^m \cdot n^{m+1}  \cdot (m\cdot\Delta(A,b))^{m^2}\cdot \Delta(A,b)^m). 
\end{equation*}

We show how to avoid a dependence of the running time
on the largest absolute value of an entry in $b$: For fixed $m$, an
integer program can be solved in time polynomially bounded by
$n$ and the largest absolute value $\Delta$ of an entry in
$A$. This result is one important ingredient to solve the optimization problem \eqref{equ:opt_problem} in polynomial-time for any constant values of $\maxsubdet(A)$, provided that $A$ has no singular $(n \times n)$-submatrices  and $\rank(A) = n$.
It turns out that the condition that all $(n\times n)$-sub-determinants shall be non-zero imposes very harsh restrictions on $A$. In particular,  $A$ can have at most $n+1$ rows provided that $n$ exceeds a certain constant.

\section{Dynamic Programming Revisited}
\subsection{The Pure Integer Case}\label{section:pure:integer:case}

We show that one can solve problem~\eqref{eq:1} in time polynomial in $n$, $\Delta$ and $\log(\max_i\{|b_i|\})$ where $\Delta=\max_{i,j}\{|A_{i,j}|\}$.
This is an improvement over Papadimitriou's approach \cite{MR677087}, as we eliminate the unary dependency on $b$. For $S\subseteq\{1,\ldots,n\}$, let $A_S$ denote the matrix stemming from
$A$ by the columns indexed by $S$. 

\begin{lemma}
	\label{lem:1}
	If the integer program~\eqref{eq:1} is feasible and bounded, there exists an optimal solution $x^* \in \Z^n$ where at least $n-m$
	components of   $x^*$ are bounded by $(m+2)\cdot (m\cdot \Delta)^m$.
	Furthermore, the  columns of $A$ corresponding to components of
	$x^*$ that are larger than $(m+2)\cdot (m\cdot \Delta)^m$ are linearly
	independent. 
\end{lemma}

The proof of this Lemma is in Appendix~\ref{append}.
Once this Lemma is shown, we have the following result. 

\begin{theorem}
	\label{thr:1}
	There exists an algorithm that solves the integer programming
	problem~\eqref{eq:1} in time bounded by $$\mathcal{O} \left( 2^{2m} \cdot \Delta^{m^3+3m^2+2m} \cdot n^{m^3+5m^2+6m+1} \right).$$
\end{theorem}
\begin{proof}
	We assume that problem is feasible and bounded.
	Let $x^\star$ be the optimal solution as defined in Lemma~\ref{lem:1} and let $S\subseteq\{1,\ldots,n\}$ be the set of  
	indices of the components of $x^*$ that are bounded by $(m+2)\cdot (m\cdot
	\Delta)^m$. 
	By $\bar S$, we denote the complement of $S$. Now, let 
	\begin{displaymath}
		b'' := \sum_{j \in S} x^*_j A_{\cdot,j} \,\, \text{ and } \, \,b' := b - b''. 
	\end{displaymath}
	It follows that $x^*_S$ is an optimal solution of the integer program
	\begin{equation}
		\label{eq:3}
		\max \left\{  \sum_{j \in S} c_j x_j \colon \sum_{j \in S} x_j A_{\cdot,j} = b'', x \in
		\Z_{\geq0}^{S}   \right\},
	\end{equation}
	and $x^*_{\overline{S}}$ is the optimal solution of the integer
	program 
	\begin{equation}
		\label{eq:4}
		\max \left\{  \sum_{j \in \overline{S}} c_j x_j \colon \sum_{j \in \overline{S}} x_j A_{\cdot,j} = b', x \in
		\Z_{\geq0}^{\overline{S}}   \right\}.  
	\end{equation}
	
	Since $\|b''\|_\infty\leq \Delta\cdot n\cdot (m+2) \cdot (m \cdot \Delta)^m$, the integer programming
	problem~\eqref{eq:3} can be solved with Papadimitriou's algorithm 
	\cite{MR677087} in
	time 
	$\mathcal{O} \big(2^m \cdot \Delta^{m^2+m} \cdot n^{m^2 + 2m + 1} \cdot m^{2m^2+m} \cdot (m \cdot\Delta)^{m^3+m^2} \big)$.
	
	Since the columns of $A_{\overline{S}}$ are linearly independent,
	$x^*_{\overline{S}}$ is the unique solution of the system of equations
	\begin{displaymath}
		\sum_{j \in \overline{S}} x_j A_{\cdot,j} = b',
	\end{displaymath}
	which can be found by using Gaussian elimination. 
	
	The algorithm starts by enumerating all possible
	\[\mathcal{O} \left(2^m \cdot \Delta^m \cdot n^m \cdot (m+2)^m \cdot (m\cdot \Delta)^{m^2}\right)\]
	vectors $b''$ and then proceeds by enumerating all $\binom{n}{m}
	= O(n^m)$ components of $x^*$ whose absolute value might be larger
	than $(m+2) \cdot (m \cdot \Delta)^m$ in the optimal solution $x^*$. Then,
	one solves the integer program~\eqref{eq:3} with Papadimitriou's
	algorithm and the integer program~\eqref{eq:4} using Gaussian
	elimination. 
	
	Altogether this yields a running time of 
	\begin{align*}
		\mathcal{O} \, \mathlarger( 2^m \cdot \Delta^m \cdot n^m \cdot (m+2)^m \cdot (m \cdot\Delta)^{m^2} \cdot n^m \cdot 2^m \\
		\Delta^{m^2+m} \cdot n^{m^2 + 2m + 1} \cdot m^{2m^2+m} \cdot (m \cdot\Delta)^{m^3+m^2}\mathlarger).
	\end{align*}
	We can assume that $m\leq n$ and obtain the running time
	$$\mathcal{O} \left( 2^{2m} \cdot \Delta^{m^3+3m^2+2m} \cdot n^{m^3+5m^2+6m+1} \right).$$
\end{proof}

\begin{proof}[Proof of Lemma~\ref{lem:1}]   
	We assume that the objective function vector $c$ is non-degenerate
	in the following sense: We suppose that $c^T y \neq0$ for each
	integral vector $y \neq0$ with $\|y\|_\infty \leq (m+2) \cdot (m\cdot \Delta)^m$. This can be
	achieved without loss of generality with standard perturbation,
	i.e., $c := c + (\varepsilon,\varepsilon^2, \ldots,\varepsilon^n)$ for $\varepsilon>0$ small.
	
	Let $x^*$ be an optimal solution and let $S\subseteq\{1,\ldots,n\}$ be a subset
	of indices for which $x^*_s \geq (m+2) \cdot (m\cdot \Delta)^m$ for each $s \in S$.
	If the columns of $A_{\cdot,S}$ are not linearly independent, then there
	exists a non-zero integral $d \in \Z^{|S|}$, $d \neq 0$, with $A_{\cdot,S}\cdot d =
	0$. We can assume that the support of $d$ fulfills $|\operatorname{supp}(d)|\le m+1$. Then, as noted in the introduction, there exists a feasible solution to $A_{\cdot,S}\cdot x =
	0$, $x\in\Z^{|S|}$, with   $\|d\|_\infty \leq (m+2) (m\cdot \Delta)^m$. Without loss of generality, 
	$c^Td>0$, using our modified objective function. But then, $(x^*_S + d, x^*_{\overline{S}})$ is a feasible
	solution with better objective function, which is a
	contradiction. Consequently, the number of components of $x^*$
	exceeding $(m+2) (m\cdot \Delta)^m$ is bounded by $m$ and the corresponding
	columns of $A$ are linearly independent. 
	
\end{proof}

\subsection{Extensions to the Mixed Integer Setting}

This section is devoted to generalizations of Lemma \ref{lem:1} and Theorem \ref{thr:1} in order to apply the idea from the previous section to mixed-integer optimization problems of the form 
\begin{equation}
	\label{eq:mip2}
	\max\left\{ c^Tx +d^Ty \colon Ax + By = b, \, x,y\geq0, \,x \in \Z^n,\, y \in \R^l\right\},
\end{equation}
where, as before, $A \in \Z^{m×n}$ with  upper bound $\Delta$
on the absolute values of $A$, $b \in \Z^m $, $c \in \Z^n$ and $d \in \Q^l$.

If we view problem (\ref{eq:mip2}) as a parametric integer problem in variables $x$ only, then Lemma \ref{lem:1} is applicable. This observation directly leads us to a mixed-integer version of  Lemma \ref{lem:1}.

\begin{lemma}
	\label{lem:2}
	If the mixed-integer program~\eqref{eq:mip2} has an optimal solution, then it
	has an optimal solution $(x^*,y^*)$ such that $x^* \in \Z^n$, where at least $n-m$
	components of   $x^*$ are bounded by $(m+2)\cdot (m\cdot \Delta)^m$.
	Furthermore, the  columns of $A$ corresponding to components of
	$x^*$ that are larger than $(m+2)\cdot (m\cdot \Delta)^m$ are linearly
	independent. 
\end{lemma}

With this Lemma, we are prepared to prove a mixed-integer version of Theorem \ref{thr:1}. In the special case when $m$ is a constant, this result gives rise to a polynomial-time algorithm for solving the mixed-integer optimization problem (\ref{eq:mip2}).

\begin{theorem}
	\label{thr:2}
	There exists an algorithm that solves the mixed-integer programming
	problem~\eqref{eq:mip2} in time  bounded by $$ \mathcal{O} \left( 2^{2m} \cdot \Delta^{m^3+3m^2+2m} \cdot n^{m^3+5m^2+6m+1} \right) \cdot \kappa(m,l,\Delta),$$ where $\kappa(m,l,\Delta)$ is the worst case running time for solving a mixed-integer optimization problem of the type (\ref{eq:mip2}) with $m$ integer variables and $l$ continuous variables. 
\end{theorem}
\begin{proof}
	Let $(x^*,y^*)$ be an optimal solution of problem (\ref{eq:mip2}) satisfying 
	Lemma~\ref{lem:2}. By $S\subseteq\{1,\ldots,n\}$ we denote the
	indices of the   components of  $x^* $ that are bounded by $(m+2)\cdot (m\cdot
	\Delta)^m$. Furthermore let 
	\begin{displaymath}
		b'' := \sum_{j \in S} x^*_j A_{\cdot,j} \,\, \text{ and } \, \,b' := b - b''. 
	\end{displaymath}
	Then, $x^*_S$ is an optimal solution of the pure integer program
	\begin{equation}
		\label{eq:5}
		\max \left\{  \sum_{j \in S} c_j x_j \colon \sum_{j \in S} x_j A_{\cdot,j} = b'', x \in
		\Z_{\geq0}^{|S|}   \right\}, 
	\end{equation}
	and $(x^*_{\overline{S}},y^*)$ is the optimal solution of the mixed-integer
	program 
	\begin{equation}
	\begin{aligned} \label{eq:6} 
		\max & \left\{ \sum_{j \in \overline{S}} c_j x_j + \sum_{j=1}^l d_j y_j \colon \right.\\
		& \left. \sum_{j \in \overline{S}} x_j A_{\cdot,j} \,+\,  \sum_{j=1}^l y_j B_{\cdot,j} = b', x \in
		\Z_{\geq0}^{|\overline{S}|},\, y \geq 0 \right\}.
	\end{aligned}
	\end{equation}
	As in the previous section, the algorithm first enumerates all possible
	\[\mathcal{O} \left(2^m \cdot \Delta^m \cdot n^m \cdot (m+2)^m \cdot (m\cdot \Delta)^{m^2}\right)\]
	vectors $b''$, which satisfy $$\|b''\|_\infty\leq \Delta \cdot n\cdot (m+2) \cdot (m \cdot \Delta)^m.$$ Next, one proceeds by enumerating all $\binom{n}{m}
	= O(n^m)$ components of $x^*$ whose absolute value might be larger
	than $(m+2) \cdot (m \cdot \Delta)^m$ in the optimal solution $x^*$. The corresponding  integer programming
	problem~\eqref{eq:5} can be solved with Papadimitriou's algorithm in
	time $\mathcal{O} \big(2^m \cdot \Delta^{m^2+m} \cdot n^{m^2 + 2m + 1} \cdot m^{2m^2+m} \cdot (m \cdot\Delta)^{m^3+m^2} \big)$. 
	Since the columns of $A_{\overline{S}}$ are linearly independent, $| \overline{S}| \leq m$.
	Therefore, the mixed-integer program (\ref{eq:6}) can be solved in time $\kappa(m,l,\Delta)$.

	Altogether, this yields the proposed running time.
\end{proof}

\section{Application to integer programs with non-zero and bounded sub-determinants}
Let us now use Theorem~\ref{thr:1} to prove that \eqref{equ:opt_problem} can be solved in polynomial-time under the assumption that $A$ has rank $n$, $A$ has no singular $(n \times n)$-submatrices and $\maxsubdet(A) \le \deltabound$ for a constant $\deltabound < \infty$. Our plan is as follows:

First, we permute the rows of $A$ and transform the permuted matrix into Hermite Normal Form to obtain the representation illustrated in (\ref{equ:HNF}). This can be accomplished in polynomial-time (cf. \cite{frumkin1976}). Then, we show that the entries of $A$ are bounded by a constant $\entrybound(\deltabound)$ depending solely on $\deltabound$ and that for $n$ large enough, $A$ can have at most $n+1$ rows. Using these results, we find an efficient algorithm for \eqref{equ:opt_problem}: We reformulate (\ref{equ:opt_problem}) as a program of the form given in (\ref{eq:1}) {with a matrix whose number of rows is bounded by a constant which only depends on $\deltabound$.} We then apply Theorem~\ref{thr:1} to get a polynomial running-time algorithm. 

By permuting the rows of $A$, we may assume without loss of generality that $\det (A_{1: n,\cdot}) \ne 0$, where $A_{1: n,\cdot}$ is the uppermost $(n\times n)$-submatrix of $A$. Let $U$ be the unimodular matrix such that $A_{1:n,\cdot}U$ is in   Hermite Normal Form (cf. \cite{Schrijver86}, Part II, Chapters 4 and 5). 
Note that as $U^{-1}\in\Z^{n\times n}$, $AU$ has the same $(n\times n)$-sub-determinants as $A$ and that $U$ can be calculated in polynomial-time (in $m$, $n$ and the encoding size of $A$) and is polynomially bounded in the size of $A$ (cf. \cite{frumkin1976} or \cite{Schrijver86}, Part II, Chapter 5.2).
After a change of variables from $x$ to $U^{-1}x$, we can assume from now on that $A$ is a lower triangular matrix with diagonal entries
$$\left(1,\dots,1,\deltabound_1,1,\dots,1,\deltabound_2,1,\dots,1,\deltabound_k,1,\dots,1,A_{n,n}\right).$$
Moreover, we have that $\deltabound_i \le \deltabound$ for all $i\in\{1,\dots,k\}$ and for all $1\le j<i\le n$ we have $A_{i,j}<A_{i,i}$.

Our next step is to simplify $A$ further.  To this end, we apply row permutations  and column permutations  iteratively: For $i\in \{1,\dots,k\}$, assume that after row and column permutations, the rows corresponding to the diagonal entries $\deltabound_{i+1},\dots,\deltabound_k$ are at positions $n-i-1,\dots,n$. 
Let $A^{\deltabound_i}$ be the row corresponding to $\delta_i$. 
Then, exchange rows $A^{\deltabound_i}$ and $A_{n-(i+1),\cdot}$ as well as the column containing $\delta_i$ with column $n-(i+1)$.
This leads to a representation of the matrix $A$ as follows:

\begin{align}\label{equ:HNF}
	A=\left[ \begin{matrix}
		1 &&&&&& ~ \\
		& \ddots &&&&& ~ \\
		&&1&&&& ~ \\
		*&\cdots&*& \deltabound_1 &&& ~ \\
		\vdots&&&\ddots& \ddots && ~ \\
		*&\cdots&\cdots&\cdots&*& \deltabound_k& ~ \\
		A_{n,1} &\cdots&\cdots&\cdots&\cdots&  A_{n,n-1} & A_{n,n} \\
		A_{n+1,1} &\cdots&\cdots&\cdots&\cdots& A_{n+1,n-1} & A_{n+1,n} \\
		\vdots & \vdots & \vdots & \vdots& \vdots& \vdots& \vdots \\
		A_{m,1} &\cdots&\cdots&\cdots&\cdots& A_{m,n-1} & A_{m,n}
	\end{matrix} \right].
\end{align}

All entries denoted by $\ast$ are numbers between $0$ and the diagonal entry of the same row.

The submatrix consisting of the first $n$ rows has a determinant that is bounded by $\deltabound$. This allows us to conclude that $|\deltabound_1|\cdots |\deltabound_k|\cdot |A_{n,n}|\le \deltabound$ and thus $k\le\log_2(\deltabound)$.

We will now make use of two Lemmas. The corresponding proofs are postponed to Appendix \ref{append}.
\begin{lemma}\label{lemma:boundedentries}
	The entries in $A$ are bounded by a constant $C(\deltabound)$ which only depends on $\deltabound$.
\end{lemma}
\begin{remark}
	Lemma \ref{lemma:boundedentries} also holds in the case where $A$ has singular $(n \times n)$ submatrices.
\end{remark}
\begin{lemma}\label{lemma:np1rows}
	Let $n> (2C(\deltabound)+1)^{\log_2\deltabound+3}+\log_2\deltabound$. Then, $A$ has at most $n+1$ rows.
\end{lemma}
We are now prepared to prove the main result of this section.
\begin{theorem}\label{thr:application}
	There exists an algorithm that solves problem (\ref{equ:opt_problem}) in time polynomially bounded by $m$, $n$, $\delta$ and the encoding size of the input data.
\end{theorem}
\begin{proof}[Proof of Theorem~\ref{thr:application}]
	If $n\leq (2C(\deltabound)+1)^{\log_2\deltabound+3}+\log_2\deltabound$, then by using Lenstra's algorithm (cf. \cite{lenstrajr1983Intprofixnumvar}), the corresponding integer program can be solved in polynomial-time. 
	
	Otherwise, $n> (2C(\deltabound)+1)^{\log_2\deltabound+3}+\log_2\deltabound$. By Lemma \ref{lemma:np1rows}, $A$ has at most $n+1$ rows. 
	Furthermore, in view of \eqref{equ:HNF}, we may assume that $A$ is of the form
	\begin{align*}
		A = 	
		\left[ \begin{matrix}
			-\mathcal{I} & {\bf0} \\
			\widetilde A & \widehat A 
		\end{matrix} \right],
	\end{align*}
	where $\mathcal{I}$ denotes the $(n-k-1)$-dimensional identity matrix, ${\bf0}=\{0\}^{(n-k-1)\times (k+1)}$, 
	\begin{align*}
		\widetilde A := 	
		\left[ \begin{matrix}
			*& \cdots & *  \\
			\vdots &  & \vdots \\
			* & \cdots & *  \\
			\alpha_{1} & \ldots &\alpha_{n-k-1} \\
			\beta_{1} & \ldots &\beta_{n-k-1} \\
		\end{matrix} \right]\in \Z^{(k+2)\times (n-k-1)}
	\end{align*}
	and
	\begin{align*}
		\widehat A :=\left[ \begin{matrix}
			\deltabound_1 & ~ & ~ & ~ \\
			\ddots & \ddots & ~& ~ \\
			\cdots & * & \deltabound_k& ~ \\
			\alpha_{n-k} & \ldots &  \alpha_{n-1} & \alpha_{n} \\
			\beta_{n-k} & \ldots &  \beta_{n-1} & \beta_{n} \\
		\end{matrix} \right]\in\Z^{(k+2)\times (k+1)}.
	\end{align*}
	It holds that $\deltabound\ge\deltabound_i > 0$ and by Lemma \ref{lemma:boundedentries}, $|A_{ij}|\leq C(\deltabound)$ for all $1\le i\le n+1$, $1\le j \le n$.\\
	
	We denote $\bar A:=\left[\begin{matrix} \widetilde A & \widehat A \end{matrix}\right]$.	For a vector $x\in\R^n$, let $\widetilde x\in\Z^{n-k-1}$ be the first $n-k-1$, $\widehat x\in\Z^{k+1}$ be the last $k+1$ components of $x$.
	Thus, $\bar A x = \widetilde A \widetilde x + \widehat A \widehat x$. We write $x = \left(\begin{matrix}
	\widetilde x\\\widehat x
	\end{matrix}\right)$.
	Similarly, we write $b = \left(\begin{matrix}
	\tilde b\\\bar b
	\end{matrix}\right)$, where $\widetilde b\in\Z^{n-k-1}$ and $\bar b\in\Z^{k+2}$, such that $Ax\le b \Leftrightarrow \widetilde x\ge \widetilde b,~\bar Ax\le\bar b$. 
	
	$Ax\le b$, $x \in \Z^n$ can then be reformulated as $\bar A x \le \bar b,$ $\tilde x \ge \widetilde b$, $x \in \Z^n$, which in turn can be restated as 
	\begin{align*}
		\bar Ay \le \ & \bar b-\bar A\binom{\widetilde b}{0},\\
		\widetilde y \ge \ & 0,
	\end{align*}
	where $y := x - \binom{\widetilde b}{0}$.
	
	This reformulation leads to the following maximization problem:
	\begin{align}\label{eq:reform1}
		\max \left\{c^Tx \colon \bar Ax \le \bar b,~x=\left(\begin{matrix}
			\widetilde x \\\widehat x
		\end{matrix}\right),~\widetilde x \ge 0,~x\in\Z^n\right\}.
	\end{align}
	To arrive at a problem of the form (\ref{eq:1}), we apply a standard technique: We introduce new variables $\widehat x_i^+:= \max \{\widehat x_i,0 \}$ and $\widehat x_i^- := \min\{\widehat x_i,0 \}$ as well as slack variables $z\ge 0$ and reformulate (\ref{eq:reform1}) as
	\begin{align}\label{eq:reform2}
		\max &~\left[\widetilde c^T~~\widehat c^T~~-\widehat c^T~~\mathbf{0}^T\right]\left[\begin{matrix}
			\widetilde{x}&
			\widehat x^+&
			\widehat x^-&
			z
		\end{matrix}\right]^T\nonumber\\
		\st &~\left[\widetilde A ~~\widehat A~~ -\widehat A~~ \mathcal{I}  \right] \left[\begin{matrix}
			\widetilde{x} &
			\widehat x^+ &
			\widehat x^- &
			z
		\end{matrix}\right]^T = \bar b,\\
		&~\widetilde x,~\widehat x^+,~\widehat x^-,~z\ge 0,\nonumber
	\end{align}
	where $c = \left(\begin{matrix}
	\widetilde c\\\widehat c
	\end{matrix}\right)$. Here, $\mathbf{0}\in\Z^{k+2}$ denotes the   $(k+2)$-dimensional zero vector and $\mathcal{I}$ is the $(k+2) \times (k+2)$-identity matrix.
	
	Recall that  $k\le \log_2(\deltabound)$. Hence, the matrix $\left[\widetilde A ~~\widehat A~~ -\widehat A~~ \mathcal{I}  \right]\in\Z^{(k+2)\times (n+2k+3)}$ in (\ref{eq:reform2}) has at most $\log_2 \deltabound+2$ rows and at most $n+2\log_2\deltabound+3$ columns. Furthermore, in view of Lemma~\ref{lemma:boundedentries}, each entry is bounded by $C(\deltabound)$. 
	We can therefore apply  Theorem~\ref{thr:1}. This gives the desired result, where the overall running time is bounded by  
	\begin{align*}
		\mathcal{O} \left(n^{(\log_2\deltabound+2)^3+5(\log_2\deltabound+2)^2+6(\log_2\deltabound+2)+1} \right).
	\end{align*}
\end{proof}


\appendix
\renewcommand*{\thesection}{\Alph{section}} 
\section{Technical Proofs}
\label{append}

\begin{proof}[Proof of Lemma~\ref{lemma:boundedentries}.] 
	We assume that $A$ is in the form \eqref{equ:HNF}.
	The first $n$ rows are part of the Hermite Normal Form. Hence, by definition, they fulfill the statement.
	Let $\alpha=[\alpha_1,\ldots,\alpha_n]$ be any other row of $A$.

	Let $B_0:=\deltabound$, $q := \lceil \log_2\deltabound\rceil$ and for $i\in\{1,\dots,q\}$, consider the increasing sequence $B_i:=\deltabound +  \sum_{l = 0}^{i-1}B_{l} \deltabound^{\log_2 \deltabound}(\log_2 \deltabound)^{(\log_2 \deltabound)/2}.$ We show that all entries of $\alpha$ are bounded by the constant $B_q$  as follows.
	
	Let $A_i$ denote the square submatrix of $A$ that consists of the first $n$ rows, except for the $i$-th row, which is replaced by $\alpha$.
	Then, since $|\det(A_n)|=\delta_1 \dots  \delta_k \cdot \alpha_n$, it follows that $|\alpha_n| \le \deltabound=B_0$.\\
	
	\noindent
	\textit{Case i)} ~ $i\ge n-k$:
		
		\noindent
		Consider $\alpha_i$ and assume that it holds that {$B_{n-j}\ge |\alpha_j|$} for all $n\ge j > i$. We can express $|\det(A_i)|$ as follows.
		\begin{align*}
		|\deltabound_1|\cdots|\deltabound_{r-1}||\det\underbrace{\left[ \begin{matrix}
			\alpha_i& &\cdots& &\alpha_n\\
			* & \deltabound_{r+1}& & \\
			* & * & \ddots & & \\
			\vdots & \vdots &* & \deltabound_k & \\
			A_{n,i} & & \dots & & A_{n,n}\\
			\end{matrix}\right]}_{=:\bar A}|.\\
		\end{align*}
		Let $\bar A^j$ be the matrix $\bar A$ without the first row and without column $j$. Then,
		\begin{align*}
		\deltabound \ge~& |\det \bar A| = \left|\alpha_i \det \bar A^1 + \sum_{l = 2}^{n-i+1}(-1)^{l+1}\alpha_{l+i-1} \det \bar A^l \right|\\
		\ge~& |\alpha_i\det\bar A^1| - \sum_{l = 2}^{n-i+1}|\alpha_{l+i-1} \det \bar A^l|,
		\end{align*}
		and thus
		\begin{align*}
		|\alpha_i| \le 	\frac{1}{|\det \bar A^1|}\left(\deltabound + \sum_{l = 2}^{n-i+1}|\alpha_{l+i-1} \det \bar A^l|\right).
		\end{align*}
		Furthermore, $1\le|\det \bar A^1 |\le \deltabound$.
		
		Since the absolute value of each entry in $\bar A^j$ is bounded by $\deltabound$, we can apply the Hadamard inequality \cite{hadamard1893Resdunquerelauxdet} to obtain
		$|\det \bar A^j|\le \deltabound^{n-i}(n-i)^{(n-i)/2}\le \deltabound^{\log_2 \deltabound}(\log_2 \deltabound)^{(\log_2 \deltabound)/2}$. This provides us with the bound 
		{\begin{align*}
			|\alpha_i| \le~& \deltabound +  \sum_{l = 2}^{n-i+1}B_{n-l-i+1} \deltabound^{\log_2 \deltabound}(\log_2 \deltabound)^{(\log_2 \deltabound)/2} =B_{n-i}\le B_k.
			\end{align*}
			Thus $|\alpha_i|\le B_{k}\le B_q$.}
		~\\
		
	\noindent
	\textit{Case ii)} ~ $i < n-k$:
		
		\noindent
		Similar to the previous case, we can express $|\det(A_i)|$ as
		\begin{align*}
		|\det\underbrace{\left[ \begin{matrix}
			\alpha_i& \alpha_{n-k}&\cdots& \cdots&\alpha_n\\
			* & \deltabound_{1}& & \\
			* & * & \ddots & & \\
			\vdots & \vdots & & \deltabound_k & \\
			A_{n,i} & A_{n,n-k}& \cdots & \cdots& A_{n,n}\\
			\end{matrix}\right]}_{=:\bar A}|.\\
		\end{align*}
		Let $\bar A^j$ be $\bar A$ without the first row and column $j$, so that
		\begin{align*}
		\deltabound \ge~& |\det \bar A| = \left|\alpha_i \det \bar A^1 + \sum_{l = 2}^{k+2}(-1)^{l+1}\alpha_{l+n-k-2} \det \bar A^l \right| \\
		\ge~& |\alpha_i\det\bar A^1| - \sum_{l = 2}^{k+2}|\alpha_{l+n-k-2} \det \bar A^l|,
		\end{align*}
		and
		\begin{align*}
		|\alpha_i| \le 	\frac{1}{|\det \bar A^1|}\left(\deltabound + \sum_{l = 2}^{k+2}|\alpha_{l+n-k-2} \det \bar A^l|\right).
		\end{align*}
		{We arrive at the bound 
			\begin{align*}
			|\alpha_i| \le \deltabound +  \sum_{l = 2}^{k+2}B_{k+2-l} \deltabound^{\log_2 \deltabound}(\log_2 \deltabound)^{(\log_2 \deltabound)/2}\le B_q.
			\end{align*}}
	This completes the proof by letting $C(\deltabound):= B_q$.
\end{proof}
\begin{proof}[Proof of Lemma~\ref{lemma:np1rows}.]
	Let $A$ be defined as illustrated in \eqref{equ:HNF}. Recall that $A$ has no singular $(n \times n)$-submatrices. For the purpose of deriving a contradiction, assume that $n>(2C(\deltabound)+1)^{\log_2\deltabound+3}+\log_2\deltabound$ and that $A$ has precisely $n+2$ rows.
	Let $\tilde A$ be the matrix $A$ without rows $i$ and $j$, where $i, j < n-k$, $i \neq j$.
	Observe that
	\begin{align*}
	|\det \widetilde A| = | \det \underbrace{
	\left[ \begin{matrix}
			& & \deltabound_1  & ~ & ~ & ~\\
			\vdots&\vdots   &  & \ddots & ~ & ~\\
			&  & & &  \deltabound_k & ~\\
			A_{n,i} &A_{n,j} &&\ldots&&  A_{n,n} \\
			A_{n+1,i} &A_{n+1,j} &&\ldots&&  A_{n+1,n} \\
			A_{n+2,i} &A_{n+2,j} &&\ldots&&  A_{n+2,n} \\
	\end{matrix}\right]}_{=:\widehat A^{ij}} |.
	\end{align*}
	
	$\widehat A^{ij}$ is a $(k+3) \times (k+3)$-matrix. Its determinant cannot be zero. This implies that the first two columns of $\widehat A^{ij}$ must be different for each choice of $i$ and $j$. 
	
	From Lemma~\ref{lemma:boundedentries}, it follows that the absolute value of any entry of $\widehat A^{ij}$  is bounded  by $C(\deltabound)$.  Therefore, the first two columns are in $\{-C(\deltabound),\dots,C(\deltabound)\}^{k+3}$. Since $k\le \log_2\deltabound$, there exist at most $(2C(\deltabound)+1)^{\log_2\deltabound+3}$ such vectors. Consequently, as $n> (2C(\deltabound)+1)^{\log_2\deltabound+3}+\log_2\deltabound$, there must exist two indices $i\neq j\in \{1,\dots,n-k-1  \}$ such that $\det(\widehat A^{ij})=0$. This contradicts that there are no singular $(n \times n)$-submatrices within $A$. The statement follows.
\end{proof}


\begin{thebibliography}{1}
\expandafter\ifx\csname url\endcsname\relax
  \def\url#1{\texttt{#1}}\fi
\expandafter\ifx\csname urlprefix\endcsname\relax\def\urlprefix{URL }\fi
\expandafter\ifx\csname href\endcsname\relax
  \def\href#1#2{#2} \def\path#1{#1}\fi

\bibitem{Schrijver86}
A.~Schrijver, {Theory of Linear and Integer Programming}, John Wiley and Sons,
  NY, 1986.

\bibitem{Bonifasetal}
N.~Bonifas, M.~{Di Summa}, F.~Eisenbrand, N.~H\"{a}hnle, M.~Niemeier, {On
  sub-determinants and the diameter of polyhedra}, Discrete \& Computational
  Geometry 52~(1) (2014) 102--115.
\newblock \href {http://dx.doi.org/10.1007/s00454-014-9601-x}
  {\path{doi:10.1007/s00454-014-9601-x}}.

\bibitem{Dyeretal}
M.~Dyer, A.~Frieze, {Random walks, totally unimodular matrices, and a
  randomised dual simplex algorithm}, Mathematical Programming 64~(1-3) (1994)
  1--16.
\newblock \href {http://dx.doi.org/10.1007/BF01582563}
  {\path{doi:10.1007/BF01582563}}.

\bibitem{veselov2009Intprobimmat}
S.~I. Veselov, A.~J. Chirkov, {Integer program with bimodular matrix}, Discrete
  Optimization 6~(2) (2009) 220--222.
\newblock \href {http://dx.doi.org/10.1016/j.disopt.2008.12.002}
  {\path{doi:10.1016/j.disopt.2008.12.002}}.

\bibitem{MR677087}
C.~H. Papadimitriou, {On the complexity of integer programming.}, J. ACM 28~(4)
  (1981) 765--768.
\newblock \href {http://dx.doi.org/10.1145/322276.322287}
  {\path{doi:10.1145/322276.322287}}.

\bibitem{frumkin1976}
M.~A. Frumkin, {An algorithm for the reduction of a matrix of integers to
  triangular form with power complexity of the computations (in Russian)},
  Ekonomika i Matematicheskie Metody 12 (1976) 173--178.

\bibitem{lenstrajr1983Intprofixnumvar}
H.~W. Lenstra, {Integer programming with a fixed number of variables},
  Mathematics of operations research 8~(4) (1983) 538--548.
\newblock \href {http://dx.doi.org/10.1287/moor.8.4.538}
  {\path{doi:10.1287/moor.8.4.538}}.

\bibitem{hadamard1893Resdunquerelauxdet}
J.~Hadamard, {R\`{e}solution d'une question relative aux d\`{e}terminants},
  Bulletin des Sciences Math\'{e}matiques 2~(17) (1893) 240--246.

\end{thebibliography}

\end{document}